\documentclass[10pt]{amsart}
\usepackage[T1]{fontenc}
\usepackage[utf8]{inputenc}
\usepackage{amsmath,amssymb,amsthm,mathrsfs}
\usepackage[a4paper, margin=1.1in]{geometry}
\usepackage{microtype}
\usepackage{enumitem}
\usepackage{xcolor}
\usepackage{todonotes}
\usepackage{aliascnt}
\usepackage[colorlinks=true,linkcolor=blue,citecolor=red,urlcolor=cyan]{hyperref}
\usepackage[nameinlink,noabbrev]{cleveref}
\usepackage{relsize}
\usepackage[bbgreekl]{mathbbol}
\usepackage{amsfonts}
\usepackage[all]{xy}

\numberwithin{equation}{section}

\newtheorem{theorem}{Theorem}[section]
\newaliascnt{proposition}{theorem}
\newtheorem{proposition}[proposition]{Proposition}
\aliascntresetthe{proposition}
\newaliascnt{lemma}{theorem}
\newtheorem{lemma}[lemma]{Lemma}
\aliascntresetthe{lemma}
\newaliascnt{corollary}{theorem}
\newtheorem{corollary}[corollary]{Corollary}
\aliascntresetthe{corollary}
\theoremstyle{definition}
\newaliascnt{hypothesis}{theorem}

\aliascntresetthe{hypothesis}
\newaliascnt{question}{theorem}

\aliascntresetthe{question}
\newaliascnt{example}{theorem}

\aliascntresetthe{example}
\theoremstyle{remark}
\newaliascnt{remark}{theorem}
\newtheorem{remark}[remark]{Remark}
\aliascntresetthe{remark}
\crefname{theorem}{theorem}{theorems}
\Crefname{theorem}{Theorem}{Theorems}
\crefname{proposition}{proposition}{propositions}
\Crefname{proposition}{Proposition}{Propositions}
\crefname{lemma}{lemma}{lemmas}
\Crefname{lemma}{Lemma}{Lemmas}
\crefname{corollary}{corollary}{corollaries}
\Crefname{corollary}{Corollary}{Corollaries}
\crefname{hypothesis}{hypothesis}{hypotheses}
\Crefname{hypothesis}{Hypothesis}{Hypotheses}
\crefname{question}{question}{questions}
\Crefname{question}{Question}{Questions}
\crefname{example}{example}{examples}
\Crefname{example}{Example}{Examples}
\crefname{remark}{remark}{remarks}
\Crefname{remark}{Remark}{Remarks}

\newcommand{\Q}{\mathbb Q}
\newcommand{\Qp}{\mathbb Q_p}
\newcommand{\Zp}{\mathbb Z_p}
\newcommand{\OL}{\mathcal O_L}

\newcommand{\an}{\mathrm{an}}
\newcommand{\Spec}{\operatorname{Spec}}

\newcommand{\Hom}{\operatorname{Hom}}
\newcommand{\End}{\operatorname{End}}

\newcommand{\Fil}{\operatorname{Fil}}
\newcommand{\GL}{\operatorname{GL}}

\newcommand{\et}{\mathrm{\acute et}}
\newcommand{\dR}{\mathrm{dR}}

\newcommand{\proket}{\mathrm{prok\acute et}}
\newcommand{\ket}{\mathrm{k\acute et}}
\newcommand{\crys}{\mathrm{crys}}

\DeclareSymbolFontAlphabet{\mathbb}{AMSb}
\DeclareSymbolFontAlphabet{\mathbbl}{bbold}

\begin{document}

\title{Geometric Detection and Finiteness of Log Crystalline Local Systems}

\author{Zhenmou Liu}
\address{School of Mathematical Sciences, University of Science and Technology of China, Hefei, Anhui 230026, PR China}
\email{zhenmouliu@mail.ustc.edu.cn}

\author{Jinbang Yang}
\address{School of Mathematical Sciences, University of Science and Technology of China, Hefei, Anhui 230026, PR China}
\email{yjb@mail.ustc.edu.cn}

\begin{abstract}
We prove a geometric detection theorem for log crystalline local systems in the smooth proper setting with a normal crossings boundary.
The associated special-fiber log $F$-isocrystal detects geometric absolute irreducibility and, modulo constant rank-one twists, determines the pullback to the geometric generic fiber.
The main ingredient is a determinant argument combining a finite-order rank-one reduction, logarithmic period comparison, and the Pl\"ucker embedding.
As an application, using Abe--Esnault's finiteness theorem and finiteness of stable lattices, we prove finiteness, up to geometrically constant rank-one twists, of geometrically absolutely irreducible log crystalline local systems of fixed rank.
Coefficients may lie in any fixed finite extension of $\Qp$, and no bound on the Hodge--Tate weights is imposed.
\end{abstract}

\maketitle
\section{Introduction}\label{sec:Introduction}

Let $K/\Qp$ be a finite extension, with ring of integers
$\mathcal O_K$ and residue field $k$.
Let $X$ be a smooth proper scheme over $\mathcal O_K$
with geometrically connected generic fiber, and let
$D\subset X$ be a relative simple normal crossings divisor.
Write $(X_K,D_K)$ for the generic fiber of $(X,D)$.
Let $(\mathcal X,\mathcal D)$ be the $p$-adic completion of
$(X,D)$, with rigid generic fiber
$(\mathcal X_K,\mathcal D_K)$ and special fiber
$(\mathcal X_k,\mathcal D_k)$.
We equip these pairs with the log structures induced
by their respective boundaries.
Since $X$ is proper over $\mathcal O_K$, there is a
canonical identification
\[
(\mathcal X_K,\mathcal D_K)
\simeq
(X_K^{\an},D_K^{\an}),
\]
where $(\cdot)^{\an}$ denotes rigid analytification.
We set
\[
\mathcal X_K^\circ
:=
\mathcal X_K-\mathcal D_K
\simeq
(X_K-D_K)^{\an}.
\]

In \cite{GR}, Guo and Reinecke developed a prismatic approach to crystalline local systems on smooth $p$-adic formal schemes, establishing an equivalence with analytic prismatic $F$-crystals and relative crystalline comparison theorems over ramified bases.
In \cite{Ino25}, Inoue extended the crystalline comparison theory to the semi-stable logarithmic setting.
Our setting is the horizontally semi-stable case $n=0$ in \cite[Definition~2.4]{Ino25}.
Since $X$ is smooth over $\mathcal O_K$, there are no vertical logarithmic components or semi-stable degenerations in the special fiber.
Nevertheless, the horizontal boundary $\mathcal D$ plays an essential role in the logarithmic crystalline realization and the relevant period comparisons.

We consider \emph{log crystalline $p$-adic local systems} on $\mathcal X_K^\circ$ in the sense arising from Inoue's logarithmic crystalline comparison theory.
Roughly speaking, such a local system is one whose logarithmic crystalline period realization is associated with a log $F$-isocrystal on $(\mathcal X_k,\mathcal D_k)$.
The precise definition and the relevant coefficient conventions will be recalled in \Cref{sec:LogCrystallineLoc}.
For a finite extension $L/\Qp$ and a log crystalline $L$-local system $\mathbb L$, we denote its associated $L$-coefficient log $F$-isocrystal by $\mathcal E_{\mathbb L}$.
We call an $L$-coefficient log $F$-isocrystal \emph{constant} if it is pulled back from $\Spec k$.

Let $\mathbb C_p$ be a completed algebraic closure of $K$, and set
$\mathcal X_{\mathbb C_p}^\circ:=\mathcal X_K^\circ\widehat\otimes_K \mathbb C_p$.
An $L$-local system on $\mathcal X_K^\circ$ is called \emph{geometrically absolutely irreducible} if its pullback to $\mathcal X_{\mathbb C_p}^\circ$ remains irreducible after every finite extension of $L$. An $\mathcal O_L$-local system is called geometrically absolutely irreducible if its rationalization is geometrically absolutely irreducible.

Our first main result shows that the associated log $F$-isocrystal retains substantial information about the geometric local system.

\begin{theorem}[Geometric detection and recovery]
\label{thm:geometric_detection_log}
Let $\mathbb L,\mathbb L_1,\mathbb L_2$ be log crystalline $L$-local systems on $\mathcal X_K^\circ$, and let $\mathcal E_{\mathbb L},\mathcal E_{\mathbb L_1},\mathcal E_{\mathbb L_2}$ be their associated log $F$-isocrystals, respectively.
\begin{enumerate}
\item
If $\mathbb L$ is geometrically absolutely irreducible, then
$\mathcal E_{\mathbb L}\otimes_L\overline{\Q}_p$
is irreducible as a $\overline{\mathbb Q}_p$-coefficient log $F$-isocrystal on $(\mathcal X_k,\mathcal D_k)$.

\item
Suppose that
\[
\mathcal E_{\mathbb L_1}\otimes_L\overline{\Q}_p
\simeq
\bigl(
\mathcal E_{\mathbb L_2}\otimes_L\overline{\Q}_p
\bigr)\otimes C
\]
for a constant rank-one log $F$-isocrystal $C$ with $\overline{\Q}_p$-coefficients.
Then the pullbacks of $\mathbb L_1$ and $\mathbb L_2$ to $\mathcal X_{\mathbb C_p}^\circ$ are isomorphic as $L$-local systems.
\end{enumerate}
\end{theorem}

The first assertion contains the main geometric argument.
A sub-log-$F$-isocrystal of $\mathcal E_{\mathbb L}\otimes_L\overline{\Q}_p$ need not itself arise from a sub-local system of $\mathbb L$.
We bypass this obstruction by passing to its determinant.
Suppose that $\mathcal E_{\mathbb L}\otimes_L\overline{\Q}_p$ admits a nonzero proper sub-log-$F$-isocrystal.
Its determinant gives a rank-one subobject in the corresponding exterior power.
After a suitable constant twist, this determinant becomes a finite-order rank-one object.
We show that such an object is realized by a crystalline local system with finite monodromy.
The logarithmic de Rham period comparison then produces a decomposable line in the corresponding exterior power of the pullback of $\mathbb L$ to $\mathcal X_{\mathbb C_p}^\circ$.
The Pl\"ucker embedding recovers from this line a proper sub-local system of that pullback, contradicting geometric absolute irreducibility.

The second assertion is simpler.
After forgetting Frobenius, a constant rank-one log $F$-isocrystal has trivial underlying logarithmic connection.
An isomorphism of two log $F$-isocrystals up to a constant rank-one twist therefore identifies their underlying logarithmic flat bundles.
The logarithmic de Rham period comparison then recovers the corresponding pullbacks of the local systems to $\mathcal X_{\mathbb C_p}^\circ$.

We next apply \Cref{thm:geometric_detection_log} to a finiteness problem.
By the finiteness theorem of Abe--Esnault \cite[Corollary~4.3]{AE}, irreducible overconvergent $F$-isocrystals of fixed rank with bounded ramification have finitely many isomorphism classes modulo constant rank-one twists.
Together with \Cref{thm:geometric_detection_log}, this yields a finiteness theorem for geometrically absolutely irreducible log crystalline local systems.

A rank-one $\mathcal O_L$-local system (resp. $L$-local system) on $\mathcal X_K^\circ$ is called \emph{geometrically constant} if it is pulled back from a rank-one $\mathcal O_L$-local system (resp. $L$-local system) on $\operatorname{Spa} K$.

\begin{theorem}[Finiteness]
\label{thm:log_finiteness}
For every positive integer $r$, the geometrically absolutely irreducible log crystalline $\mathcal O_L$-local systems of rank $r$ on $\mathcal X_K^\circ$ have finitely many isomorphism classes modulo tensor products with geometrically constant rank-one $\mathcal O_L$-local systems.
\end{theorem}

The rational finiteness follows from the corresponding finiteness theorem for overconvergent $F$-isocrystals together with \Cref{thm:geometric_detection_log}.
Finiteness of stable integral lattices then upgrades the rational statement to the integral statement.

Fixing the determinant removes the remaining arithmetic twists and gives the following consequence.

\begin{corollary}
\label{cor:fixed_determinant}
There are finitely many isomorphism classes of geometrically absolutely irreducible log crystalline $\mathcal O_L$-local systems of rank $r$ on $\mathcal X_K^\circ$ with fixed determinant.
\end{corollary}

The finiteness problem considered here is closely related to the work of Krishnamoorthy--Yang--Zuo \cite{KYZ1,KYZ2}, which is based on the Fontaine--Faltings framework.
In \cite[Theorem~1.2]{KYZ1}, the residual representation is required to be geometrically absolutely irreducible and $r\leq p-2$, while \cite[Theorem~1.2]{KYZ2} assumes unramified base and coefficient fields and restricts the Hodge--Tate weights to an interval $[a,a+p-1]$.
Our result allows ramified base and coefficient fields, and imposes no restriction on the rank or Hodge--Tate weights.
Unlike the irreducibility comparison in \cite[Lemma~2.3]{KYZ2}, our approach uses finite-order determinant twists, logarithmic period comparison, and the Pl\"ucker embedding to detect geometric irreducibility directly from the associated log $F$-isocrystal.

Although our horizontal logarithmic setting is simpler than the general semi-stable setting of \cite{Ino25}, the boundary remains essential in the crystalline realization, period comparison, and finiteness argument.
Our results may therefore be viewed as a step toward analogous questions in genuinely semi-stable logarithmic settings.

The paper is organized as follows.
In \Cref{sec:LogCrystallineLoc}, we recall the definition and basic properties of log crystalline local systems.
In \Cref{sec:realizations}, we establish the necessary logarithmic realizations and coefficient conventions.
In \Cref{sec:detection}, we prove the finite-order realization for rank-one objects and establish the geometric detection and recovery theorem.
In \Cref{sec:lattices}, we study the finiteness of stable lattices and arithmetic twists.
Finally, in \Cref{sec:finiteness}, we combine these results with the finiteness theorem of Abe--Esnault to prove \Cref{thm:log_finiteness} and \Cref{cor:fixed_determinant}.

\subsection*{Acknowledgments}

The authors acknowledge the use of an LLM for assistance with language editing and phrasing.
The authors gratefully acknowledge the financial support provided by the Key Program (Grant No. 12331002) and the International Collaboration Fund (Grant No. W2441003) of the National Natural Science Foundation of China, and the Fundamental Research Funds for the Central Universities.

\section{Log crystalline local systems}\label{sec:LogCrystallineLoc}

We keep the notation from the Introduction.
Let $\mathcal X_{K,\log}$ be the fs log adic space obtained by endowing $\mathcal X_K$ with the divisorial log structure associated to $\mathcal D_K$.
Write $(\mathcal X_{K,\log})_{\ket}$ and $(\mathcal X_{K,\log})_{\proket}$ for the Kummer étale site and the pro-Kummer-étale site of $\mathcal X_{K,\log}$, respectively; see \cite[Definitions~4.1.16 and~5.1.2]{DLLZ23a}.
We work with Kummer étale $\mathbb Z_p$-local systems on $\mathcal X_{K,\log}$ in the sense of \cite[Definition~6.3.1]{DLLZ23a}.
Let
\[
j:\mathcal X_K^\circ\hookrightarrow\mathcal X_{K,\log}
\]
denote the natural open immersion.
By \cite[Corollary~6.3.4]{DLLZ23a}, restriction along $j$ yields an equivalence between the category of Kummer étale $\mathbb Z_p$-local systems on $\mathcal X_{K,\log}$ and the category of $\mathbb Z_p$-local systems on $\mathcal X_K^\circ$.
We shall use this equivalence without further comment.

We say that a $\mathbb Z_p$-local system $\mathbb T$ on $\mathcal X_K^\circ$ is \emph{log crystalline} if there exists a locally free log $F$-isocrystal $(\mathcal E,\varphi_{\mathcal E})$ on $(\mathcal X_k,\mathcal D_k)$ together with a Frobenius-compatible $\mathbb B_{\crys}$-linear isomorphism
\[
\mathbb T\otimes_{\mathbb Z_p}\mathbb B_{\crys}
\xrightarrow{\sim}
\mathbb B_{\crys}(\mathcal E)
\]
on $(\mathcal X_{K,\log})_{\proket}$.
Here $\mathbb B_{\crys}$ and $\mathbb B_{\crys}(\mathcal E)$ are the crystalline period sheaves on $(\mathcal X_{K,\log})_{\proket}$ defined in \cite[Lemma~4.2 and Construction~4.17]{Ino25}, respectively.
In the terminology of \cite[Definition~4.18(1)]{Ino25}, this is a weakly semi-stable local system in the present horizontally semi-stable setting.

Let $(E,\nabla_E)$ be the vector bundle with integrable logarithmic connection on $\mathcal X_K$ underlying $\mathcal E$; see \cite[Proposition~3.26]{Ino25}.
By \cite[Proposition~4.20]{Ino25}, there exists a unique separated and exhaustive decreasing filtration $\Fil^\bullet E$ with locally free graded pieces such that $(\mathcal E,\varphi_{\mathcal E},\Fil^\bullet E)$ is a filtered log $F$-isocrystal.
In particular, $\Fil^\bullet E$ satisfies Griffiths transversality, and the induced $\mathbb B_{\dR}$-linear isomorphism
\[
\mathbb T\otimes_{\mathbb Z_p}\mathbb B_{\dR}
\xrightarrow{\sim}
\mathbb B_{\dR}(\mathcal E)
\]
is compatible with filtrations.
Here $\mathbb B_{\dR}$ and $\mathbb B_{\dR}(\mathcal E)$ are the de Rham period sheaves on $(\mathcal X_{K,\log})_{\proket}$ defined in \cite[Lemma~4.2 and Construction~4.17]{Ino25}, respectively.
Under the logarithmic Riemann--Hilbert correspondence of \cite[\S3.2]{DLLZ23b}, the filtered vector bundle with integrable logarithmic connection associated with $\mathbb T[1/p]$ is $(E,\nabla_E,\Fil^\bullet E)$; see also \cite[Propositions~4.19 and~4.20]{Ino25}.

Let $L/\mathbb Q_p$ be a finite extension with ring of integers $\mathcal O_L$.
An \emph{$\mathcal O_L$-local system} (resp. \emph{$L$-local system}) on $\mathcal X_K^\circ$ is a $\mathbb Z_p$-local system $\mathbb T$ (resp. $\mathbb Q_p$-local system $\mathbb L$) on $\mathcal X_K^\circ$ together with a homomorphism of $\mathbb Z_p$-algebras (resp. $\mathbb Q_p$-algebras)
\[
\mathcal O_L\longrightarrow\operatorname{End}_{\mathbb Z_p}(\mathbb T)
\qquad
\text{(resp. }
L\longrightarrow\operatorname{End}_{\mathbb Q_p}(\mathbb L)\text{)}.
\]

We say that an $\mathcal O_L$-local system $\mathbb T$ is log crystalline if its underlying $\mathbb Z_p$-local system is log crystalline. An $L$-local system is said to be log crystalline if it admits a log crystalline $\mathcal O_L$-lattice.

\section{Log crystalline realizations and coefficient conventions}
\label{sec:realizations}

This section fixes the realization, coefficient, and Frobenius conventions used in the proof of \Cref{thm:geometric_detection_log}.
We use the definition of log $F$-isocrystals on the log crystalline site given in \cite[Definition~3.4]{Ino25}.
In our smooth logarithmic setting, the associated logarithmic connections of locally free log $F$-isocrystals are convergent and have nilpotent residues; see \cite[Remarks~A.5--A.6]{KP22}.
We shall therefore also use their realizations as convergent log $F$-isocrystals on the log convergent site.

The logarithmic crystalline realizations used below are naturally $\Qp$-linear, whereas later we also use finite extensions of $\Qp$, as well as the $\overline{\Q}_p$-coefficient categories occurring in \cite{Abe18,AE}.
We first explain how these conventions are related.

\subsection{Finite coefficient fields}
Let $L/\Qp$ be a finite extension.
An \emph{$L$-coefficient log $F$-isocrystal} on
$(\mathcal X_k,\mathcal D_k)$ is a log $F$-isocrystal
$\mathcal E$ with $\Qp$-coefficients together with a homomorphism
of $\Qp$-algebras
\[
L\longrightarrow
\operatorname{End}_{F\text{-}\operatorname{Isoc}^{\log}
(\mathcal X_k,\mathcal D_k)}(\mathcal E).
\]
Morphisms are morphisms of log $F$-isocrystals compatible with the
$L$-actions.
We denote the resulting category by
\[
F\text{-}\operatorname{Isoc}^{\log}
(\mathcal X_k,\mathcal D_k)\otimes_{\mathbb Q_p}L.
\]
If $(\mathcal E,\varphi_{\mathcal E})$ is an $L$-coefficient log
$F$-isocrystal on $(\mathcal X_k,\mathcal D_k)$, we denote by
$(E_{\mathcal E},\nabla_{\mathcal E})$ its associated vector bundle
with integrable logarithmic connection on $\mathcal X_K$, as in
\cite[Proposition~3.26]{Ino25}.
Thus $E_{\mathcal E}$ is a locally free
$\mathcal O_{\mathcal X_K}\otimes_{\Qp}L$-module. The \emph{rank of an $L$-coefficient log $F$-isocrystal $\mathcal E$} is defined to be the rank of $E_{\mathcal E}$ as an $\mathcal O_{\mathcal X_K}\otimes_{\mathbb Q_p} L$-module, which we denote by $\operatorname{rank}_L(\mathcal E)$. The underlying $\mathbb Q_p$-coefficient log $F$-isocrystal has rank $[L:\mathbb Q_p]\operatorname{rank}_L(\mathcal E)$.

If $L'/L$ is a finite extension, extension of coefficients defines
a functor
\[
-\otimes_L L':
F\text{-}\operatorname{Isoc}^{\log}
(\mathcal X_k,\mathcal D_k)\otimes_{\mathbb Q_p}L
\longrightarrow
F\text{-}\operatorname{Isoc}^{\log}
(\mathcal X_k,\mathcal D_k)\otimes_{\mathbb Q_p}L',
\]
which sends $\mathcal E$ to
$\mathcal E\otimes_L L'$. This functor preserves the rank, i.e.,
\[\operatorname{rank}_{L}(\mathcal E)=\operatorname{rank}_{L'}(\mathcal E\otimes_L L')=\frac{1}{[L'\colon L]}\operatorname{rank}_{L}(\mathcal E\otimes_L L').\]

Let $\mathbb L$ be a log crystalline $L$-local system.
By forgetting the $L$-action, we obtain a log crystalline
$\mathbb Q_p$-local system with an associated log
$F$-isocrystal $\mathcal E$.
The $L$-action on $\mathbb L$ induces an action on
$\mathbb B_{\crys}(\mathcal E)$ through the crystalline
comparison.
By \cite[Proposition~4.22]{Ino25}, this action uniquely
descends to $\mathcal E$.
We denote the resulting $L$-coefficient log $F$-isocrystal
by $\mathcal E_{\mathbb L}$.
This construction is compatible with tensor operations
and extension of coefficients.

We use the analogous notion of overconvergent $F$-isocrystals with
coefficients in a finite extension $L/\Qp$; see
\cite{Abe18,AE}. Let $\mathcal X_k^\circ:=\mathcal X_k-\mathcal D_k$.
We denote the corresponding category by
\[
F\text{-}\operatorname{Isoc}^{\dagger}
(\mathcal X_k^\circ)\otimes_{\mathbb Q_p}L.
\]

\subsection{Algebraic \texorpdfstring{$p$}{p}-adic coefficients}
We define
\[
F\text{-}\operatorname{Isoc}^{\log}
(\mathcal X_k,\mathcal D_k)\otimes_{\mathbb Q_p}\overline{\Q}_p
:=
2\text{-}\varinjlim_{L/\Qp}
\left(
F\text{-}\operatorname{Isoc}^{\log}
(\mathcal X_k,\mathcal D_k)\otimes_{\mathbb Q_p}L
\right),
\]
where $L$ runs through the finite extensions of $\Qp$ contained
in $\overline{\Q}_p$ and the transition functor is given by extension of coefficients.

For every finite extension $L/\Qp$, the coefficient-extension
functors above induce a natural functor
\[
-\otimes_L\overline{\Q}_p:
F\text{-}\operatorname{Isoc}^{\log}
(\mathcal X_k,\mathcal D_k)\otimes_{\mathbb Q_p}L
\longrightarrow
F\text{-}\operatorname{Isoc}^{\log}
(\mathcal X_k,\mathcal D_k)
\otimes_{\mathbb Q_p}\overline{\Q}_p,
\]
which sends $\mathcal E$ to $\mathcal E\otimes_L\overline{\Q}_p$.
Every object, morphism, subobject, or isomorphism with
$\overline{\Q}_p$-coefficients is therefore defined over some
finite extension of $\Qp$, and finitely many such data may be
simultaneously defined over a common finite coefficient field.

We use the same coefficient convention for overconvergent
$F$-isocrystals on $\mathcal X_k^\circ$. Set
\[
F\text{-}\operatorname{Isoc}^{\dagger}
(\mathcal X_k^\circ)\otimes_{\mathbb Q_p}\overline{\Q}_p
:=
2\text{-}\varinjlim_{L/\Qp}
\left(
F\text{-}\operatorname{Isoc}^{\dagger}
(\mathcal X_k^\circ)\otimes_{\mathbb Q_p}L
\right).
\]
This is the coefficient convention used when applying results from
\cite{Abe18,AE}.

\begin{lemma}
\label{lem:log_extension_coefficients}
Let $L/\Qp$ be a finite extension.
Then restriction to $\mathcal X_k^\circ$ induces a fully faithful functor
\[
F\text{-}\operatorname{Isoc}^{\log}
(\mathcal X_k,\mathcal D_k)\otimes_{\mathbb Q_p}L
\longrightarrow
F\text{-}\operatorname{Isoc}^{\dagger}
(\mathcal X_k^\circ)\otimes_{\mathbb Q_p}L.
\]
Its essential image consists precisely of those objects having unipotent local monodromy.
Moreover, this functor preserves and reflects irreducibility.
The same assertions hold with $\overline{\Q}_p$-coefficients.
\end{lemma}

\begin{proof}
We first show that every object in the essential image has unipotent local monodromy.
By \cite[Theorem~6.4.5]{Ked07}, the restriction functor from convergent log isocrystals with nilpotent residues on $(\mathcal X_k,\mathcal D_k)$ to overconvergent isocrystals on $\mathcal X_k^\circ$ is fully faithful, and its essential image consists precisely of those having unipotent local monodromy.
Since log $F$-isocrystals have nilpotent residues by \cite[Remark~7.1.4]{Ked07}, their restrictions have unipotent local monodromy.

We then show that the essential image of the restriction functor consists precisely of those objects having unipotent local monodromy.
Let $\mathcal E^\circ$ be an overconvergent $L$-coefficient $F$-isocrystal on $\mathcal X_k^\circ$ with unipotent local monodromy.
By \cite[Theorem~6.4.5]{Ked07}, its underlying overconvergent isocrystal extends to a convergent log isocrystal $\mathcal E$ with nilpotent residues.
By full faithfulness of the restriction functor, the $L$-action on $\mathcal E^\circ$ extends uniquely to an action
\[
L\longrightarrow\operatorname{End}(\mathcal E).
\]
Since Frobenius pullback preserves nilpotent residues, the Frobenius structure on $\mathcal E^\circ$ also extends uniquely to an isomorphism
\[
\varphi_{\mathcal E}:F^*\mathcal E\xrightarrow{\sim}\mathcal E.
\]
By faithfulness, $\varphi_{\mathcal E}$ is compatible with the $L$-action.
Thus $\mathcal E$ is an $L$-coefficient log $F$-isocrystal.
Full faithfulness for $L$-coefficient $F$-isocrystals follows from full faithfulness of the underlying restriction functor.

We next show that restriction preserves and reflects irreducibility.
Let $\mathcal E$ be an $L$-coefficient log $F$-isocrystal, and suppose that
\[
0\neq\mathcal E_1^\circ
\subsetneq
\mathcal E|_{\mathcal X_k^\circ}
\]
is an overconvergent $L$-coefficient $F$-subisocrystal.
Since unipotent local monodromy is preserved under subobjects, $\mathcal E_1^\circ$ admits an $L$-coefficient log $F$-isocrystal extension $\mathcal E_1$.
By full faithfulness, the inclusion extends uniquely to a morphism $\mathcal E_1\to\mathcal E$.
Since restriction is exact and faithful, this morphism is a nonzero proper monomorphism.
Hence irreducibility of $\mathcal E$ implies irreducibility of its restriction.
The converse follows from the exactness and faithfulness of restriction.

Finally, the assertions for $\overline{\Q}_p$-coefficients follow by passing to the $2$-inductive limit over finite extensions $L/\Qp$, since objects, morphisms, and finite-rank subobjects are defined over some finite extension.
\end{proof}

\subsection{Frobenius and coefficient conventions}
Put
\[
K_0=W(k)[1/p],
\qquad
q=|k|=p^f,
\]
and let $\sigma$ denote the Witt-vector Frobenius on $K_0$.
Our convention for log $F$-isocrystals uses the absolute Frobenius together with the induced $\sigma$-action on
$K_0$.
The convention used in \cite{Abe18,AE} is slightly different:
there one uses the $f$-th iterate of the absolute Frobenius and works within the corresponding $K_0$-linear category.
The lemma below compares these two conventions under the assumption that the coefficient field contains $K_0$.

\begin{lemma}[Change of Frobenius conventions]
\label{lem:frobenius-convention}
Let $L/K_0$ be a finite extension.
Consider the following two categories:
\begin{enumerate}
\item
 log $F$-isocrystals with a commuting
$\Qp$-algebra action of $L$;

\item
 log $F^f$-isocrystals with a commuting
$K_0$-algebra action of $L$.
\end{enumerate}
Then the categories in \textup{(1)} and \textup{(2)} are naturally equivalent.
The same holds for overconvergent $F$-isocrystals on
$\mathcal X_k^\circ$.
\end{lemma}

\begin{proof}
Since the constructions are compatible with extension of coefficients from $K_0$ to $L$, it suffices to treat the case $L=K_0$. Let $a$ be a primitive element for $K_0$ over $\mathbb Q_p$.

Let $\mathcal E$ be an object of \textup{(1)}.
For each $0\leq j\leq f-1$, let
\[
\mathcal E_j
=
\left\{
x\in\mathcal E
\;\middle|\;
a x=\sigma^j(a)x
\right\},
\]
where on the right-hand side $K_0$ acts through the structural
$K_0$-action on $\mathcal E$.
The decomposition
\[
K_0\otimes_{\Qp}K_0
\simeq
\prod_{j=0}^{f-1}K_0
\]
then gives
$\mathcal E = \bigoplus_{j=0}^{f-1}\mathcal E_j$.
The Frobenius cyclically permutes these components.
Hence its $f$-fold iterate preserves $\mathcal E_0$ and induces an
isomorphism
\[
(F^*)^f\mathcal E_0
\xrightarrow{\sim}
\mathcal E_0,
\]
which defines an
object of \textup{(2)}.

Conversely, let $\mathcal E'$ be an object of \textup{(2)}.
Set
\[
\mathcal E
=
\bigoplus_{j=0}^{f-1}(F^*)^j\mathcal E',
\]
where the coefficient action of $K_0$ on the $j$-th summand is
transported through $\sigma^j$.
The absolute Frobenius identifies successive summands, while the
$q$-power Frobenius of $\mathcal E'$ identifies the last summand with
the first.
In this way $\mathcal E$ becomes an object of \textup{(1)}.

These two constructions are quasi-inverse.
The same argument applies to overconvergent $F$-isocrystals on
$\mathcal X_k^\circ$.
\end{proof}

Finally, if $C$ is a constant $L$-coefficient log $F$-isocrystal, namely the pullback of
a rank-one $L$-coefficient log $F$-isocrystal on $\Spec k$, then its logarithmic flat
realization is trivial:
\[
(E_{\mathcal E_C},\nabla_{\mathcal E_C})
\simeq
\bigl(
\mathcal O_{\mathcal X_K}\otimes_{\Qp}L,
\mathrm d\otimes 1
\bigr).
\]
Only this triviality of the logarithmic connection will be used below; we do not assume that $C$ itself lies in the essential image of the log crystalline realization.

\section{Geometric detection and recovery}
\label{sec:detection}

We now prove \Cref{thm:geometric_detection_log}.
The proof uses three ingredients: a degree-zero geometric Hom comparison, a finite-order realization for rank-one log $F$-isocrystals, and an exact logarithmic period realization for subobjects.

\subsection{Degree-zero period comparison}\label{subsec:Deg0Com}
Fix a geometric point $\bar x$ of $\mathcal X_{\mathbb C_p}^\circ$, and set
\[
\Gamma:=\pi_1^{\et}(\mathcal X_K^\circ,\bar x),
\qquad
H:=\pi_1^{\et}(\mathcal X_{\mathbb C_p}^\circ,\bar x),
\qquad
G_K:=\operatorname{Gal}(\overline K/K).
\]
As $\mathcal X_K^\circ$ is geometrically connected, we have the fundamental exact sequence
\[
1\longrightarrow H
\longrightarrow\Gamma
\longrightarrow G_K
\longrightarrow 1.
\]
Let $B_{\dR}$ denote Fontaine's de Rham period field, and let $\mathbb B_{\dR}$ stand for the de Rham period sheaf on the pro-Kummer \'etale site.

\begin{lemma}[Geometric Hom comparison]
\label{lem:hom-period}
Let $L/\Qp$ be a finite extension and let $\mathbb L_1,\mathbb L_2$ be $L$-local systems on $\mathcal X_K^\circ$.
Then there is a natural isomorphism
\[
\Hom_{L[H]}
\bigl(
(\mathbb L_1)_{\bar x},
(\mathbb L_2)_{\bar x}
\bigr)
\otimes_{\mathbb Q_p}B_{\dR}
\xrightarrow{\sim}
H^0\left(
(\mathcal X_{\mathbb C_p,\log})_{\proket},
\mathcal Hom_L(\mathbb L_1,\mathbb L_2)
\otimes_{\mathbb Q_p}\mathbb B_{\dR}
\right).
\]
It is compatible with composition, identity maps, tensor product, and evaluation at $\bar x$.
\end{lemma}

\begin{proof}
Put
\[
\mathcal F:=\mathcal Hom_L(\mathbb L_1,\mathbb L_2).
\]
Via the equivalence between Kummer \'etale local systems on $\mathcal X_{K,\log}$ and local systems on $\mathcal X_K^\circ$, we regard $\mathcal F$ as a Kummer \'etale $\mathbb Q_p$-local system after forgetting the $L$-action.
The underlying continuous $\mathbb Q_p$-representation has compact image, and hence admits a stable $\mathbb Z_p$-lattice $\mathcal F_0$.
Thus $\mathcal F=\mathcal F_0[1/p]$.
Applying \cite[Lemma~3.6.1]{DLLZ23b} to $\mathcal F_0$ and using that $p$ is invertible in $B_{\dR}^+$ and $\mathbb B_{\dR}^+$ gives
\begin{equation}\label{equ:1}
    H^0_{\ket}
\bigl(
\mathcal X_{\mathbb C_p,\log},
\mathcal F
\bigr)
\otimes_{\mathbb Q_p}B_{\dR}^+
\xrightarrow{\sim}
H^0\left(
(\mathcal X_{\mathbb C_p,\log})_{\proket},
\mathcal F\otimes_{\mathbb Q_p}\mathbb B_{\dR}^+
\right),
\end{equation}
where $\mathcal X_{\mathbb C_p,\log}$ denotes the base change of $\mathcal X_{K,\log}$ to $\mathbb C_p$. Since $\mathcal X_{\mathbb C_p,\log}$ is quasi-compact and
$\mathcal F\otimes_{\mathbb Q_p}\mathbb B_{\dR}^+$ is
$t$-torsion-free, we have
\begin{equation}\label{equ:2}
    H^0\left(
(\mathcal X_{\mathbb C_p,\log})_{\proket},
\mathcal F\otimes_{\mathbb Q_p}\mathbb B_{\dR}^+
\right)[1/t]
\xrightarrow{\sim}
H^0\left(
(\mathcal X_{\mathbb C_p,\log})_{\proket},
\mathcal F\otimes_{\mathbb Q_p}\mathbb B_{\dR}
\right).
\end{equation}
Indeed, let $g$ be a global section of
$\mathcal F\otimes_{\mathbb Q_p}\mathbb B_{\dR}$.
Choose a finite covering $\{U_i\}$ such that
$g|_{U_i}=g_i/t^{n_i}$, where $g_i$ is a local section of
$\mathcal F\otimes_{\mathbb Q_p}\mathbb B_{\dR}^+$ on $U_i$.
Let $n=\max_i n_i$.
Then the sections $t^{n-n_i}g_i$ have the same images in
$\mathcal F\otimes_{\mathbb Q_p}\mathbb B_{\dR}$ on the overlaps.
Since the natural map
\[
\mathcal F\otimes_{\mathbb Q_p}\mathbb B_{\dR}^+
\longrightarrow
\mathcal F\otimes_{\mathbb Q_p}\mathbb B_{\dR}
\]
is injective by $t$-torsion-freeness, they already agree before inverting $t$ and hence glue to a global section
$g'$ of $\mathcal F\otimes_{\mathbb Q_p}\mathbb B_{\dR}^+$.
By construction, $g'/t^n=g$.
This proves \eqref{equ:2}.

Inverting $t$ in \eqref{equ:1} and combining it with \eqref{equ:2}, we obtain
\[
H^0_{\ket}
\bigl(
\mathcal X_{\mathbb C_p,\log},
\mathcal F
\bigr)
\otimes_{\mathbb Q_p}B_{\dR}
\xrightarrow{\sim}
H^0\left(
(\mathcal X_{\mathbb C_p,\log})_{\proket},
\mathcal F\otimes_{\mathbb Q_p}\mathbb B_{\dR}
\right).
\]
Applying \cite[Corollary~6.3.4]{DLLZ23a} over $\mathbb C_p$ and inverting $p$, we obtain
\[
H^0_{\ket}
\bigl(
\mathcal X_{\mathbb C_p,\log},
\mathcal F
\bigr)
\simeq
H^0_{\et}
\bigl(
\mathcal X_{\mathbb C_p}^\circ,
\mathcal F
\bigr).
\]
The latter is naturally identified with
\[
\mathcal F_{\bar x}^{H}
=
\Hom_{L[H]}
\bigl(
(\mathbb L_1)_{\bar x},
(\mathbb L_2)_{\bar x}
\bigr).
\]
This gives the desired isomorphism.
\end{proof}

\subsection{Finite-order rank-one objects}
We first record a finite-order realization result for ordinary rank-one $F$-isocrystals.
We then apply it to rank-one log $F$-isocrystals after a suitable constant twist.

\begin{lemma}
\label{lem:finite-order-isocrystal}
Let $L/K_0$ be a finite extension and let $\mathcal U$ be a rank-one $L$-coefficient $F$-isocrystal on $\mathcal X_k$.
Assume that there exists an integer $n\geq 1$ such that
$\mathcal U^{\otimes_L n}\simeq \mathbf 1_L$.
Then there exists a crystalline $L$-local system $\chi$ on $\mathcal X_K$ such that
\[
\mathcal E_{\chi}\simeq\mathcal U.
\]
Moreover, the $L$-representation associated with $\chi$ has finite image, and the associated filtered flat bundle of $\chi$ carries the trivial filtration.
\end{lemma}

\begin{proof}
We first show that the underlying $K_0$-coefficient
$F$-isocrystal is unit-root in the sense of
\cite[p.~433]{Crew}.
We work with the $q$-Frobenius convention by
\Cref{lem:frobenius-convention}.
For every closed point $x\in\mathcal X_k$, the relation
$\mathcal U^{\otimes_L n}\simeq\mathbf 1_L$ implies that
all Frobenius slopes of $\mathcal U_x$ vanish, since
the Frobenius slopes of a rank-one object are multiplied
by $n$ under its $n$-th tensor power.
Hence the underlying $K_0$-coefficient $F$-isocrystal
of $\mathcal U$ is unit-root.

By the unit-root correspondence of \cite[Theorem~1.3]{Crew}, the underlying $K_0$-coefficient $F$-isocrystal corresponds to a continuous $K_0$-representation $V_0$ of $\pi_1^{\et}(\mathcal X_k)$ of rank $[L:K_0]$.
Since this correspondence is fully faithful, the action
$L\longrightarrow\operatorname{End}_{K_0}(\mathcal U)$
induces an action
$L\longrightarrow
\operatorname{End}_{\pi_1^{\et}(\mathcal X_k)}(V_0).$
Hence $V_0$ is naturally an $L$-representation of rank-one.
It is therefore given by a character
\[
\chi_0:
\pi_1^{\et}(\mathcal X_k)
\longrightarrow
L^\times.
\]
such that $\chi_0^{\otimes n}\simeq 1$. In particular, $\chi_0$ has finite image.

Let $f_k:\mathcal Y_k\longrightarrow\mathcal X_k$ be the finite \'etale Galois cover corresponding to $\ker(\chi_0)$. Let $$G:=\operatorname{Gal}(\mathcal Y_k/\mathcal X_k)\cong \operatorname{Im}\chi_0.$$
Then $\chi_0$ factors uniquely through a character $\overline{\chi}_0:G\longrightarrow L^\times$. The cover $f_k$ lifts to a finite \'etale Galois cover $f:\mathcal Y\longrightarrow\mathcal X$ by \cite[\href{https://stacks.math.columbia.edu/tag/04DZ}{Tag~04DZ}]{Stacks}. Passing to rigid generic fibers gives a finite \'etale Galois cover $f_K:\mathcal Y_K\longrightarrow\mathcal X_K$ with Galois group $G$.

Let $e_{\overline{\chi}_0}\in L[G]$ be the idempotent corresponding to the character $\overline{\chi}_0:G\to L^\times$. By the unit-root correspondence, we have $\mathcal U
\simeq
e_{\overline{\chi}_0}\bigl(f_{k,*}\mathbf 1_L\bigr)$. Thus by \cite[{Theorem~B}]{GR}, $\chi
:=
e_{\overline{\chi}_0}\bigl(f_{K,*}L\bigr)$ is crystalline with $\mathcal E_{\chi}\cong \mathcal U$.

Finally, the filtered flat bundle associated with the constant $L$-local system on $\mathcal Y_K$ carries the trivial filtration. Since finite \'etale pushforward and passage to direct summands preserve the filtration, the filtered flat bundle associated with $\chi$ also carries the trivial filtration.
\end{proof}

We now apply the preceding lemma to logarithmic rank-one objects.
The constant twist introduced below need not itself lie in the essential image of the log crystalline realization.

\begin{lemma}[Finite-order realization]
\label{lem:finite-order}
Let $\mathcal E$ be a rank-one $L$-coefficient log $F$-isocrystal on $(\mathcal X_k,\mathcal D_k)$.
After passing to a finite extension $L'/L$, there exist a constant rank-one $L'$-coefficient log $F$-isocrystal $C$ and a crystalline $L'$-local system $\chi$ on $\mathcal X_K$ such that
\[
\mathcal E_{\chi^\circ}
\simeq
(\mathcal E\otimes_L L')\otimes_{L'} C,
\]
where $\chi^\circ:=\chi|_{\mathcal X_K^\circ}$.
Moreover, the $L'$-linear representation attached to $\chi$ has finite image, and its associated filtered flat bundle carries the trivial filtration.
\end{lemma}

\begin{proof}
After replacing $L'$ by a further finite extension if necessary, we may assume that $K_0\subset L'$.
We then use \Cref{lem:frobenius-convention} to pass to the $q$-Frobenius convention.
The restriction
\[
(\mathcal E\otimes_L L')|_{\mathcal X_k^\circ}
\] of $\mathcal E\otimes_L L'$ to $\mathcal X_k^\circ$
is a rank-one overconvergent $F$-isocrystal with $L'$-coefficients; see \cite[Theorem~4.5.5]{Ked07}.
By \cite[Lemma~6.1(i)]{Abe18} and \cite[Remark~4.6]{AE}, after further enlarging $L'$ there exists a constant rank-one $L'$-coefficient overconvergent $F$-isocrystal $C$ such that
\[
\left(
(\mathcal E\otimes_L L')|_{\mathcal X_k^\circ}
\otimes_{L'} C
\right)^{\otimes_{L'} n}
\simeq
\mathbf 1_{L'}
\]
for some $n>0$.
We also denote by $C$ the corresponding constant rank-one log $F$-isocrystal on $(\mathcal X_k,\mathcal D_k)$.
By the full faithfulness of the restriction functor from log $F$-isocrystals on $(\mathcal X_k,\mathcal D_k)$ to overconvergent $F$-isocrystals on $\mathcal X_k^\circ$, the above isomorphism uniquely lifts to
\[
\mathcal U^{\otimes_{L'} n}
\simeq
\mathbf 1_{L'},
\qquad
\mathcal U:=
(\mathcal E\otimes_L L')\otimes_{L'}C.
\]

Since $\mathcal U$ has finite tensor order, its logarithmic residues vanish.
Indeed, the residue of $\mathcal U^{\otimes n}$ is $n$ times the residue of $\mathcal U$, while the tensor unit has zero residue.
Hence $\mathcal U$ comes from an ordinary $L'$-coefficient $F$-isocrystal on $\mathcal X_k$ and the statement follows from \Cref{lem:finite-order-isocrystal}.
\end{proof}

\subsection{Period realization of subobjects}
For a log $F$-isocrystal $\mathcal E$ on $(\mathcal X_k,\mathcal D_k)$, let $\mathbb B_{\dR}(\mathcal E)$ denote its logarithmic de Rham period evaluation as in \cite[Construction~4.17]{Ino25}.
If $\mathcal E$ carries an $L$-action, functoriality gives $\mathbb B_{\dR}(\mathcal E)$ the induced commuting $L$-action.

By \cite[Construction~4.17]{Ino25}, the formation of $\mathbb B_{\dR}(\mathcal E)$ is locally obtained by evaluating $\mathcal E$ on the corresponding logarithmic crystalline thickening and extending scalars to $\mathbb B_{\dR}$.
Hence it is exact on locally free log $F$-isocrystals and is compatible with tensor products, duals, determinants, and exterior powers.
If $\mathbb L$ is a log crystalline $L$-local system on $\mathcal X_K^\circ$, the defining crystalline comparison induces a natural tensor-compatible, $L\otimes_{\mathbb Q_p}\mathbb B_{\mathrm{dR}}$-linear isomorphism
\[
\mathbb B_{\dR}(\mathcal E_{\mathbb L})
\simeq
\mathbb L\otimes_{\mathbb Q_p}\mathbb B_{\dR}.
\]

In particular, if $\mathcal E'\subset\mathcal E$ is an inclusion of $L$-coefficient log $F$-isocrystals and $\mathcal E'$ has coefficient rank $s$, then
\[
\mathbb B_{\dR}(\det_L\mathcal E')
\longrightarrow
\bigwedge_{L\otimes_{\mathbb Q_p}\mathbb B_{\dR}}^s
\mathbb B_{\dR}(\mathcal E)
\]
identifies $\mathbb B_{\dR}(\det_L\mathcal E')$ with the determinant line of the rank-$s$ subbundle $\mathbb B_{\dR}(\mathcal E')\subset\mathbb B_{\dR}(\mathcal E)$.
Thus this line is locally generated by a decomposable vector.

\subsection{Proof of the detection and recovery theorem}
We first record an elementary descent lemma used in the proof of part~\textup{(2)}.

\begin{lemma}[Descent of isomorphisms]
\label{lem:descent}
Let $L$ be an infinite field, $A$ a nonzero commutative
$L$-algebra, and $V_1,V_2$ finite-dimensional
$L$-representations of a group $H$.
Then the following statements hold.
\begin{enumerate}
\item
There is a natural isomorphism
\[
\Hom_{L[H]}(V_1,V_2)\otimes_L A
\simeq
\Hom_{A[H]}(V_1\otimes_L A,V_2\otimes_L A).
\]

\item
If $V_1\otimes_L A\simeq V_2\otimes_L A$
as $A$-representations of $H$, then
$V_1\simeq V_2$ as $L$-representations of $H$.
\end{enumerate}
\end{lemma}

\begin{proof}
For (1), since $V_1$ and $V_2$ are finite-dimensional,
the condition of being $H$-equivariant is given by
finitely many linear equations on $\Hom_L(V_1,V_2)$.
As $A$ is flat over $L$, extension of scalars preserves
the solution space of these equations.
This gives the desired isomorphism.

For (2), suppose that $V_1\otimes_L A\simeq V_2\otimes_L A$.
Then $V_1$ and $V_2$ have the same dimension.
Choose an $L$-basis $f_1,\ldots,f_r$ of $\Hom_{L[H]}(V_1,V_2)$.
By \textup{(1)}, an $H$-equivariant isomorphism over $A$
can be written as $\sum_{i=1}^r a_i f_i$
for some $a_i\in A$.
Consequently, the polynomial
\[
P(X_1,\ldots,X_r)
=
\det\left(\sum_{i=1}^r X_i f_i\right)
\]
is nonzero.
Since $L$ is infinite, there exist $b_1,\ldots,b_r\in L$
such that $P(b_1,\ldots,b_r)\ne 0$.
Thus $\sum_{i=1}^r b_i f_i$ is an
$H$-equivariant isomorphism over $L$.
\end{proof}

We now prove \Cref{thm:geometric_detection_log}.

\begin{proof}[Proof of \Cref{thm:geometric_detection_log}]

\emph{Proof of \textup{(1)}.}
Suppose that there is a nonzero proper subobject
\[
0\neq
\mathcal E'
\subsetneq
\mathcal E_{\mathbb L}\otimes_L\overline{\Q}_p
\]
of rank $s$.
By the definition of the $\overline{\Q}_p$-coefficient category, after a finite extension $L'/L$, the object $\mathcal E'$ and its inclusion are defined over $L'$.
Put $\mathbb L_{L'}:=\mathbb L\otimes_LL'$.
Thus, after replacing $\mathcal E'$ by its descent to $L'$, we have an inclusion
$\mathcal E'
\hookrightarrow
\mathcal E_{\mathbb L_{L'}}$.
Taking the determinant over $L'$ gives an injection
\[
\det_{L'}\mathcal E'
\hookrightarrow
\bigwedge_{L'}^s\mathcal E_{\mathbb L_{L'}}.
\]
We freely enlarge $L'$ during the argument.
By \Cref{lem:finite-order}, there exist a constant rank-one $L'$-coefficient log $F$-isocrystal $C$ and a crystalline $L'$-local system $\chi$ on $\mathcal X_K$ such that
\[
\mathcal E_{\chi^\circ}
\simeq
(\det_{L'}\mathcal E')\otimes_{L'} C,
\]
where $\chi^\circ:=\chi|_{\mathcal X_K^\circ}$.
Tensoring the determinant morphism by $C$ and
applying the functor $\mathbb B_{\dR}(\cdot)$, we obtain an injective $L'\otimes_{\mathbb Q_p} \mathbb B_{\dR}$-linear morphism
\[
\alpha:
\chi^\circ\otimes_{\mathbb Q_p}\mathbb B_{\dR}
\simeq
\mathbb B_{\dR}(\mathcal E_{\chi^\circ})
\hookrightarrow
\mathbb B_{\dR}
\left(
\bigwedge_{L'}^s\mathcal E_{\mathbb L_{L'}}
\right)
\otimes_{L'\otimes_{\mathbb Q_p} \mathbb B_{\dR}}
\mathbb B_{\dR}(C)
\simeq
\left(
\bigwedge_{L'}^s\mathbb L_{L'}
\right)
\otimes_{\mathbb Q_p}\mathbb B_{\dR}
\]
on $(\mathcal X_{K,\log})_{\proket}$. Here the first isomorphism is induced by the defining logarithmic crystalline comparison for $\chi^\circ$, while the second follows from the tensor and exterior-power compatibility of $\mathbb B_{\dR}(\cdot)$ together with the triviality of the logarithmic flat realization of the constant object $C$.

Since $\chi^\circ$ and $\bigwedge_{L'}^s\mathbb L_{L'}$ are finite locally free $L'$-local systems, there is a natural isomorphism
\[
\mathcal Hom_{L'\otimes_{\Qp}\mathbb B_{\dR}}(\chi^\circ\otimes_{\Qp}\mathbb B_{\dR},(\bigwedge_{L'}^s\mathbb L_{L'})\otimes_{\Qp}\mathbb B_{\dR})
\simeq \mathcal Hom_{L'}(\chi^\circ,\bigwedge_{L'}^s\mathbb L_{L'})\otimes_{\Qp}\mathbb B_{\dR}.
\]
Indeed, this can be checked locally, where both sides identify with the sheaf of matrices over $L'\otimes_{\Qp}\mathbb B_{\dR}$.
Thus the restriction of $\alpha$ to $(\mathcal X_{\mathbb C_p,\log})_{\proket}$ determines a global section
\[
\alpha\in
H^0\left(
(\mathcal X_{\mathbb C_p,\log})_{\proket},
\mathcal Hom_{L'}
\left(
\chi^\circ,
\bigwedge_{L'}^s\mathbb L_{L'}
\right)
\otimes_{\Qp}\mathbb B_{\dR}
\right).
\]

Set
\[
M:=
\Hom_{L'[H]}
\left(
\chi^\circ_{\bar x},
\bigwedge_{L'}^s\mathbb L_{L',\bar x}
\right).
\]
Choose a nonzero vector $e\in\chi^\circ_{\bar x}$.
Since $\chi^\circ_{\bar x}$ is one-dimensional over $L'$,
evaluation at $e$ gives an injective $L'$-linear map
\[
M\longrightarrow
\bigwedge_{L'}^s\mathbb L_{L',\bar x},
\qquad
f\longmapsto f(e).
\]
We identify $M$ with its image under this map.
For every $f\in M$ and $h\in H$, the $H$-equivariance of $f$ gives
\[
h\bigl(f(e)\bigr)
=
f(he)
=
\chi^\circ(h)f(e).
\]
Thus $H$ acts on $M$ through the character
$\chi^\circ|_H$.
In particular, every one-dimensional $L'$-subspace
of $M$ is $H$-stable.
Consequently, every point of $\mathbb P(M)$
corresponds to an $H$-stable line in
$\bigwedge_{L'}^s\mathbb L_{L',\bar x}$.

Consider the closed projective $L'$-scheme
\[
Z=
\mathbb P(M)\cap
\operatorname{Gr}(s,\mathbb L_{L',\bar x})
\subseteq
\mathbb P\left(
\bigwedge_{L'}^s\mathbb L_{L',\bar x}
\right),
\]
where the Grassmannian is embedded by the Pl\"ucker embedding.
By \Cref{lem:hom-period}, the morphism $\alpha$ corresponds to
an element
\[
\alpha'\in
M\otimes_{\mathbb Q_p}B_{\dR}
\simeq
M\otimes_{L'}(L'\otimes_{\mathbb Q_p}B_{\dR}).
\]
Since $\alpha$ is a locally split injection with decomposable
image, $\alpha'$ defines an
$(L'\otimes_{\mathbb Q_p}B_{\dR})$-point of $Z$.
In particular, $Z$ is nonempty.
As $Z$ is of finite type over $L'$, it has a closed point
defined over some finite extension $L''/L'$.
This point determines an $H$-stable decomposable line in $\bigwedge_{L''}^s\mathbb L_{L'',\bar x}$.
By the Pl\"ucker embedding, it corresponds to an
$H$-stable $s$-dimensional subspace of
$\mathbb L_{L'',\bar x}$.
This contradicts the geometric absolute irreducibility
of $\mathbb L$.
Therefore $\mathcal E_{\mathbb L}\otimes_L\overline{\Q}_p$
is irreducible.

\medskip

\noindent
\emph{Proof of \textup{(2)}.}
Suppose that
\[
\mathcal E_{\mathbb L_1}\otimes_L\overline{\Q}_p
\simeq
\left(
\mathcal E_{\mathbb L_2}\otimes_L\overline{\Q}_p
\right)\otimes C
\]
for a constant rank-one log $F$-isocrystal $C$.
After a finite extension $L'/L$, both the isomorphism and $C$ are defined over $L'$.
Put $\mathbb L_{i,L'}:=\mathbb L_i\otimes_L L'$.
Since the logarithmic flat realization of $C$ is trivial, the given isomorphism induces an isomorphism between the logarithmic flat bundles underlying $\mathcal E_{\mathbb L_{1,L'}}$ and $\mathcal E_{\mathbb L_{2,L'}}$.
Passing to logarithmic de Rham period realizations gives an
$L'\otimes_{\Qp}\mathbb B_{\dR}$-linear isomorphism
\[
\mathbb L_{1,L'}\otimes_{\Qp}\mathbb B_{\dR}
\xrightarrow{\sim}
\mathbb L_{2,L'}\otimes_{\Qp}\mathbb B_{\dR}.
\]
By the same argument as in the proof of \textup{(1)}, this isomorphism determines an invertible global section
\[
s\in H^0\left(
(\mathcal X_{\mathbb C_p,\log})_{\proket},
\mathcal Hom_{L'}(\mathbb L_{1,L'},\mathbb L_{2,L'})
\otimes_{\Qp}\mathbb B_{\dR}
\right).
\]
By \Cref{lem:hom-period} and \Cref{lem:descent}(1), $s$ corresponds to an element
\[
s'\in \Hom_{L'[H]}(\mathbb L_{1,L',\bar x},\mathbb L_{2,L',\bar x})\otimes_{\Qp}B_{\dR}
\cong \Hom_{L[H]}(\mathbb L_{1,\bar x},\mathbb L_{2,\bar x})\otimes_L(L'\otimes_{\Qp}B_{\dR}).
\]
Since \Cref{lem:hom-period} is compatible with composition and identity maps,
the inverse of $s$ corresponds to an inverse of $s'$.
Thus $s'$ is invertible.
Applying \Cref{lem:descent}(2) with $A=L'\otimes_{\Qp}B_{\dR}$,
we obtain an $H$-equivariant $L$-linear isomorphism
\[
\mathbb L_{1,\bar x}\simeq\mathbb L_{2,\bar x}.
\]
Consequently, $\left.\mathbb L_1\right|_{\mathcal X_{\mathbb C_p}^\circ}
\simeq
\left.\mathbb L_2\right|_{\mathcal X_{\mathbb C_p}^\circ}.$
No irreducibility hypothesis is needed for this assertion.
\end{proof}

\section{Stable lattices and arithmetic twists}
\label{sec:lattices}

In this section, we first record two purely representation-theoretic results, formulated for general compact groups and their subgroups. We will then apply them to the arithmetic and geometric \'etale fundamental groups introduced in the introduction.

\begin{proposition}[Finiteness of stable lattices]\label{prop:lattices}
Let $L/\mathbb Q_p$ be a finite extension.
Let $\Gamma$ be a compact topological group, let $H\subseteq\Gamma$ be a closed subgroup, and let $V$ be a finite-dimensional continuous $L$-representation of $\Gamma$. Suppose that $V|_H$ is absolutely irreducible. Then there are only finitely many isomorphism classes of $\Gamma$-stable full $\OL$-lattices.
\end{proposition}

\begin{proof}
We first claim that $V$ admits a $\Gamma$-stable full $\OL$-lattice.
Choose any full $\OL$-lattice $\Lambda_0'\subset V$.
Since the action of $\Gamma$ on $V$ is continuous, the stabilizer
\[
\Gamma_{\Lambda_0'}
=
\{g\in\Gamma\mid g\Lambda_0'=\Lambda_0'\}
\]
is an open subgroup of $\Gamma$.
Since $\Gamma$ is compact, $\Gamma_{\Lambda_0'}$ has finite index in $\Gamma$.
Hence the orbit $\{g\Lambda_0'\mid g\in\Gamma\}$ is finite.
It follows that
\[
\Lambda_0:=\sum_{g\in\Gamma}g\Lambda_0'
\]
is a well-defined full $\OL$-lattice in $V$, and it is clearly $\Gamma$-stable.
This proves the claim.

We now prove finiteness.
Let $\varpi$ be a uniformizer of $\mathcal O_L$.
As $V|_H \otimes_L \overline{\mathbb Q}_p$ is absolutely irreducible, $\overline{\mathbb Q}_p[\rho(H)]$ acts irreducibly on $V_{\overline{\mathbb Q}_p} := V \otimes_L \overline{\mathbb Q}_p$.
Burnside's irreducibility theorem then yields
\[
\overline{\mathbb Q}_p[\rho(H)] = \operatorname{End}_{\overline{\mathbb Q}_p}(V_{\overline{\mathbb Q}_p}).
\]
A dimension comparison therefore gives $$L[\rho(H)] = \operatorname{End}_L(V).$$
Since $\Lambda_0$ is $H$-stable, we have $\OL[\rho(H)]\subseteq\End_{\OL}(\Lambda_0)$,
and the preceding equality implies that $\OL[\rho(H)]$ spans $\End_L(V)$ over $L$.
Therefore, there exists an integer $N\geq 0$ such that
\[\varpi^N\End_{\OL}(\Lambda_0)\subseteq \OL[\rho(H)]\subseteq \End_{\OL}(\Lambda_0).
\]

Let $\Lambda\subset V$ be any $\Gamma$-stable full $\OL$-lattice.
After replacing $\Lambda$ by $\varpi^m\Lambda$ for a suitable integer $m$, we may assume that $\Lambda\subseteq\Lambda_0$ and $\Lambda\not\subseteq\varpi\Lambda_0$.
Choose $v\in\Lambda-\varpi\Lambda_0\subseteq \Lambda_0-\varpi\Lambda_0$.
Since $\OL$ is a DVR and $\Lambda_0$ is a finite free $\OL$-module, $v$ extends to an $\OL$-basis of $\Lambda_0$.
Hence $\End_{\OL}(\Lambda_0)v=\Lambda_0$.
It follows that
\[
\varpi^N\Lambda_0
=
\varpi^N\End_{\OL}(\Lambda_0)v
\subseteq
\OL[\rho(H)]v
\subseteq
\Lambda.
\]
Thus every $\Gamma$-stable full lattice is isomorphic to one satisfying $\varpi^N\Lambda_0\subseteq\Lambda\subseteq\Lambda_0$.
Since $\Lambda_0/\varpi^N\Lambda_0$ is finite, there are only finitely many $\OL$-submodules of $\Lambda_0$ containing $\varpi^N\Lambda_0$. The result follows.
\end{proof}

\begin{lemma}[Rational arithmetic twists]
\label{lem:twists}
Let $\Gamma$ be a topological group, let $H\subseteq\Gamma$ be a normal subgroup,
and let $L/\mathbb Q_p$ be a finite extension.
Let $V_1$ and $V_2$ be two continuous finite-dimensional $L$-representations of $\Gamma$
such that $V_1|_H$ and $V_2|_H$ are absolutely irreducible.
Suppose that $V_1|_H\simeq V_2|_H$.
Then there exists a unique continuous character
$\chi:\Gamma/H\to L^\times$ such that
\[
V_2\simeq V_1\otimes_L\chi.
\]
\end{lemma}

\begin{proof}
By \Cref{lem:descent}(1), we have
\[
\Hom_{L[H]}(V_1,V_2)\otimes_L\overline{\mathbb Q}_p
\simeq
\Hom_{\overline{\mathbb Q}_p[H]}
\bigl(
V_1\otimes_L\overline{\mathbb Q}_p,
V_2\otimes_L\overline{\mathbb Q}_p
\bigr).
\]
Since $V_1|_H$ and $V_2|_H$ are absolutely irreducible and isomorphic, the right-hand side is one-dimensional over $\overline{\mathbb Q}_p$ by Schur's lemma.
Hence $\Hom_{L[H]}(V_1,V_2)$ is one-dimensional over $L$.

Choose an $H$-equivariant isomorphism $f\colon V_1\xrightarrow{\sim}V_2$.
Let $\rho_i:\Gamma\to\GL_L(V_i)$ denote the representations.
For every $\gamma\in\Gamma$, since $H$ is normal in $\Gamma$, the map
$\rho_2(\gamma)f\rho_1(\gamma)^{-1}:V_1\to V_2$ is again
$H$-equivariant.
Hence there is a unique $\chi(\gamma)\in L^\times$ such that
\[
\rho_2(\gamma)f\rho_1(\gamma)^{-1}
=
\chi(\gamma)f.
\]

It follows immediately from the defining identity that
$\chi:\Gamma\to L^\times$ is a character.
Since $f$ is $H$-equivariant, $\chi|_H=1$, and hence $\chi$ factors
through a character $\chi:\Gamma/H\to L^\times$.
Moreover, the definition of $\chi$ is independent of the choice of
$f$, since $\Hom_{L[H]}(V_1,V_2)$ is one-dimensional over $L$.
Thus $\chi$ is unique, and the defining identity gives
\[
V_2\simeq V_1\otimes_L\chi.
\]
Finally, the continuity of $\chi$ follows from the continuity of
$\rho_1$ and $\rho_2$.
\end{proof}

We return to the geometric setting of the Introduction and retain the notation $\Gamma$ and $H$ introduced in \Cref{subsec:Deg0Com}.
Applying the preceding representation-theoretic results to these groups gives the following consequence.

\begin{corollary}
\label{cor:rational-twists}
Let $\mathbb L_1$ and $\mathbb L_2$ be geometrically absolutely irreducible
$L$-local systems on $\mathcal X_K^\circ$.
Then their pullbacks to $\mathcal X_{\mathbb C_p}^\circ$ are isomorphic
if and only if
\[
\mathbb L_2\simeq \mathbb L_1\otimes_L\chi
\]
for some rank-one geometrically constant $L$-local system $\chi$ on
$\mathcal X_K^\circ$.
\end{corollary}

\begin{proof}
Let $\Gamma$ and $H$ be as above.
If the pullbacks of $\mathbb L_1$ and $\mathbb L_2$ to
$\mathcal X_{\mathbb C_p}^\circ$ are isomorphic, then the corresponding
$L$-representations satisfy $V_1|_H\simeq V_2|_H$.
By \Cref{lem:twists}, there exists a continuous character
$\chi:G_K=\Gamma/H\to L^\times$ such that
$V_2\simeq V_1\otimes_L\chi$.
Equivalently,
$\mathbb L_2\simeq\mathbb L_1\otimes_L\chi$.

Conversely, if
$\mathbb L_2\simeq\mathbb L_1\otimes_L\chi$
for such a $\chi$, then $\chi$ is trivial on $H$, so its pullback to
$\mathcal X_{\mathbb C_p}^\circ$ is trivial.
Hence the pullbacks of $\mathbb L_1$ and $\mathbb L_2$ are isomorphic.
\end{proof}

\section{Finiteness}
\label{sec:finiteness}
We now combine \Cref{thm:geometric_detection_log} with the finiteness theorem of Abe--Esnault and the representation-theoretic results of \Cref{sec:lattices}.

\begin{proof}[Proof of \Cref{thm:log_finiteness}]
Let $\mathbb T$ be a geometrically absolutely irreducible log crystalline $\mathcal O_L$-local system of rank $r$ on $\mathcal X_K^\circ$.
Put $\mathbb L:=\mathbb T\otimes_{\mathcal O_L}L$.
By \Cref{thm:geometric_detection_log}\textup{(1)}, $\mathcal E_{\mathbb L}\otimes_L\overline{\Q}_p$ is irreducible.
By \Cref{lem:log_extension_coefficients}, its restriction to $\mathcal X_k^\circ$ is an irreducible overconvergent $F$-isocrystal with unipotent local monodromy.
By \cite[Theorem~4.2]{AE}, after a constant rank-one twist,
it admits an $\ell$-adic companion.
Since $\mathcal X_k$ is proper, every morphism from a smooth curve
to $\mathcal X_k^\circ$ extends to its smooth compactification.
Thus the pullback of the $F$-isocrystal has vanishing Swan conductors,
and so does its $\ell$-adic companion by the argument in the proof of
\cite[Proposition~2.8]{AE}.
Hence the companion has ramification bounded by $\mathcal D_k$.

By \cite[Corollary~4.3]{AE}, there are only finitely many
isomorphism classes of irreducible rank-$r$ overconvergent
$F$-isocrystals with $\overline{\Q}_p$-coefficients on
$\mathcal X_k^\circ$ whose $\ell$-adic companions have
ramification bounded by $\mathcal D_k$, modulo constant
rank-one twists.
In particular, only finitely many of these classes arise from geometrically absolutely irreducible log crystalline $L$-local systems of rank $r$ on $\mathcal X_K^\circ$.
For each such class, choose one geometrically absolutely irreducible log crystalline $L$-local system realizing it.
Denote these representatives by $\mathbb L_1,\ldots,\mathbb L_m$.
Thus, for every $\mathbb L$ under consideration, there exist $j$ and a constant rank-one overconvergent $F$-isocrystal $C$ such that
\[
\left.(\mathcal E_{\mathbb L}\otimes_L\overline{\Q}_p)\right|_{\mathcal X_k^\circ}
\simeq
\left.(\mathcal E_{\mathbb L_j}\otimes_L\overline{\Q}_p)\right|_{\mathcal X_k^\circ}
\otimes C.
\]
Regarding $C$ as a constant rank-one log $F$-isocrystal on $(\mathcal X_k,\mathcal D_k)$, the full faithfulness of \Cref{lem:log_extension_coefficients} gives
\[
\mathcal E_{\mathbb L}\otimes_L\overline{\Q}_p
\simeq
(\mathcal E_{\mathbb L_j}\otimes_L\overline{\Q}_p)\otimes C.
\]
By \Cref{thm:geometric_detection_log}\textup{(2)}, the pullbacks of $\mathbb L$ and $\mathbb L_j$ to $\mathcal X_{\mathbb C_p}^\circ$ are isomorphic as $L$-local systems.
Since they are geometrically absolutely irreducible, \Cref{cor:rational-twists} implies that
\[
\mathbb L\simeq\mathbb L_j\otimes_L\chi
\]
for a geometrically constant rank-one $L$-local system $\chi$.

Since $\chi$ corresponds to a continuous character of the compact group $G_K$, its image is contained in $\mathcal O_L^\times$.
Thus $\chi$ admits a geometrically constant rank-one $\mathcal O_L$-lattice $\chi_{\mathcal O_L}$.
Under the preceding isomorphism, $\mathbb T\otimes_{\mathcal O_L}\chi_{\mathcal O_L}^{-1}$ is a stable $\mathcal O_L$-lattice in $\mathbb L_j$.
By \Cref{prop:lattices}, the geometric absolute irreducibility of $\mathbb L_j$ implies that $\mathbb L_j$ admits only finitely many isomorphism classes of stable $\mathcal O_L$-lattices.
Consequently, for each $j$, only finitely many isomorphism classes of $\mathbb T$ occur modulo tensor products with geometrically constant rank-one $\mathcal O_L$-local systems.
Since there are only finitely many choices of $j$, the assertion follows.
\end{proof}

\begin{proof}[Proof of \Cref{cor:fixed_determinant}]
By \Cref{thm:log_finiteness}, there are only finitely many isomorphism classes modulo geometrically constant rank-one $\OL$-local systems.
Choose a representative $\mathbb T_j$ from each class meeting the fixed-determinant locus.
If $\mathbb T\simeq\mathbb T_j\otimes_{\OL}\chi$ and $\det_{\OL}\mathbb T\simeq\det_{\OL}\mathbb T_j$, then $\chi^{\otimes r}\simeq\mathbf 1_{\OL}$.
Thus $\chi$ corresponds to a continuous character $G_K\to\mu_r(L)$.

By local class field theory, the reciprocity map $K^\times\to G_K^{\mathrm{ab}}$ induces an isomorphism $\widehat{K^\times}\simeq G_K^{\mathrm{ab}}$.
Since $\mu_r(L)$ is a finite abelian group, we obtain
\[
\Hom_{\mathrm{cont}}(G_K,\mu_r(L))
\simeq
\Hom_{\mathrm{cont}}(K^\times,\mu_r(L))
\simeq
\Hom(K^\times/(K^\times)^r,\mu_r(L)).
\]
The group $K^\times/(K^\times)^r$ is finite, since $K/\Qp$ is a finite extension.
Hence each geometrically constant twist class contains only finitely many objects with the prescribed determinant.
\end{proof}

\begin{remark}[Why geometric absolute irreducibility is essential]
\label{rem:filtrations}
The geometric absolute irreducibility hypothesis in \Cref{thm:log_finiteness} is essential and cannot be weakened, even to geometric absolute semisimplicity in rank two over $\Qp$.
Indeed, take $K=\Qp$, $X=\Spec\Zp$, and $D=\varnothing$.
For each positive integer $n$, consider the rank-one filtered $\varphi$-module
\[
D_n=\Qp e_n,
\qquad
\varphi(e_n)=(1+pn)e_n,
\]
with filtration
\[
\Fil^iD_n=
\begin{cases}
D_n,&i\leq0,\\
0,&i\geq1.
\end{cases}
\]
Since $1+pn\in\Zp^\times$, we have $t_H(D_n)=t_N(D_n)=0$, so $D_n$ is weakly admissible.
By the Colmez--Fontaine theorem \cite{CF}, it corresponds to a rank-one unramified crystalline $\Qp$-local system, which admits a stable $\Zp$-lattice $\mathbb T'_n$.
Set $\mathbb T_n=\mathbb T'_n\oplus(\mathbb T'_n)^\vee$.
Then each $\mathbb T_n$ is crystalline with trivial determinant.
Since the geometric fundamental group of $\Spec\Qp$ is trivial, every $\mathbb T_n$ is geometrically absolutely semisimple. We claim that the family $\{\mathbb T_n\}_{n\geq 1}$ is pairwise non-isomorphic modulo geometrically constant rank-one twists.

Suppose that $\mathbb T_n\simeq\mathbb T_m\otimes\chi$ for a geometrically constant rank-one $\Zp$-local system $\chi$.
Comparing determinants gives $\chi^{\otimes2}\simeq\mathbf 1$.
Since $\mathbb T_n$ and $\mathbb T_m$ are unramified, so is $\chi$.
Hence $\chi$ is crystalline by \cite[Corollary~9.3.2]{BC09}.
Comparing the Frobenius eigenvalues of the associated filtered $\varphi$-modules yields
\[
\{1+pn,(1+pn)^{-1}\}
=
\{1+pm,(1+pm)^{-1}\}
\quad\text{or}\quad
\{-(1+pm),-(1+pm)^{-1}\}.
\]
This forces $n=m$.
Therefore the $\mathbb T_n$ are pairwise non-isomorphic modulo geometrically constant rank-one twists.
This shows that geometric absolute semisimplicity is insufficient for the finiteness conclusion of \Cref{thm:log_finiteness}.
\end{remark}

\end{document}